\documentclass{birkjour}

 \theoremstyle{definition}
 
 \theoremstyle{remark}

 \numberwithin{equation}{section}
 \theoremstyle{definition}
 \newtheorem{definition}{Definition}[section]
 \newtheorem{theorem}{Theorem}
 \newtheorem{lemma}[theorem]{Lemma}
 \theoremstyle{remark} 
 \newtheorem{remark}{Remark}
 \newtheorem{example}{Example}
 \newtheorem{corollary}[theorem]{Corollary}

 \usepackage{graphicx}
\usepackage{amsmath}
\usepackage[english]{babel}  
\usepackage{microtype}       

\begin{document}

%
%
%
%
%
%
%
%
%

\title[Approximating fixed point by Krasnoselskij iteration]
 {Approximating the common fixed point of enriched interpolative  matkowski type mapping in Banach space}

\author{Akansha Tyagi}

\address{Department of Mathematics, Faculty of Mathematical Sciences
Guru Tegh Bahadur Road, University of Delhi, Delhi 110007, India}
\email{akanshatyagi0107@gmail.com ,atyagi1@maths.du.ac.in}

\author{ Sachin Vashistha}
\address{Department of Mathematics, Hindu College, University of Delhi,Delhi 110007, India}
\email{sachin.vashistha1@gmail.com}
\date{today}

\begin{abstract}
In the  Normed space theory,
the existence of fixed points  is  
one of the main tools in improving efficiency of iterative algorithms in optimization, numerical analysis and various mathematical applications .This study introduces and investigates a recent concept termed "Enriched interpolative  Matkowski-type mapping". Building upon the well-established foundation of Matkowski-type contractions  .This extension incorporates an interpolative enrichment mechanism, yielding a refined framework for analyzing contraction mappings. The proposed concept is motivated by the desire to enhance the convergence behavior and applicability of contraction mapping principles in various mathematical and scientific domains.
\end{abstract}

\maketitle

\section{Introduction}
Fixed point theory is a basic concept that finds applications across many mathematical areas .The Banach Contraction Principle \cite{bib5} stands as a foundational theorem, renowned for its capacity to ensure both the existence and uniqueness of fixed points within the domain of contraction mappings residing in a metric space and extended to complete metric spaces by Caccioppoli 
\cite{bib6} After that Kannan established a fixed point theorem \cite{bib7,bib8} which shows that a discontinuous map can also have a fixed point.
Karapinar established the generalized Kannan type contraction and started the interpolative technique
\cite{bib9}.
Numerous Researcher have contributed to the development of interpolative fixed-point theory, including those who studied interpolative Ćirić-Reich-Rus\cite{bib11} , Hardy rogers type contraction \cite{bib12}.Additional current research on interpolative fixed-point outcomes can be found in \cite{bib13,bib14,bib15,bib16}.
Interpolative Boyd-Wong type contraction and Matkowski type contraction were concepts that were first introduced by Karpinar et al. \cite{bib17} in 2020. \\
Contrarily, in two  papers \cite{bib18,bib19}, the first author introduced the technique of enriching nonexpansive mappings. This method was later successfully applied to the class of strictly pseudo-contractive mappings \cite{bib19} .
These findings serve as the foundation for this article's goal, which is to use the technique of enriching interpolative matkowski type mappings to the class of interpolative matkowski  mappings to produce enriched interpolative matkowski mappings. We establish a fixed point theorem for this new class of mappings and demonstrate that, as opposed to the Picard iteration, a suitable Krasnoselskij iteration can be used to approximate their fixed points.\\
Motivated by recent studies, this paper aims to define enriched interpolative Matkowski mappings for complete normed linear spaces. Subsequently, we will present a proof establishing the existence of  fixed point for this type of mapping. Within the scope of our paper, we will demonstrate that interpolative Matkowski mapping is a particular case of enriched interpolative Matkowski mapping.
In the end, we will give an application of enriched interpolative Matkowski mapping in split convex feasibility problem.\\
\section{Preliminaries}
\begin{definition} \cite{bib1,bib28} 
Let $\mathrm{Z}$ be a collection of functions $\zeta \colon [0, \infty) \rightarrow [0, \infty)$ that satisfy the condition.
    \item $\zeta$ is non-decreasing,
    \item $\lim_{n \to \infty} \zeta^n(t) = 0$ for all $t > 0$.
\end {definition}

\begin{definition}\cite{bib4}
Let $(G, \|\cdot\|)$ be a complete normed linear space. A self-mapping $R: G \to G$ is called an interpolative  Matkowski mapping if there exist $a, b, c \in (0, 1)$ satisfying $a + b + c < 1$  and $\zeta \in \mathrm{Z}$ such that\\
\begin{equation}\label{2}
\begin{aligned}
\|  Rp - Rq \| 
\leq &\zeta\bigg[ \| p - q\|^b \|p - Rp\|^a \|q - Rq\|^c\\
& \bigg( \frac{1}{2} ( \| p- Rq\| + \|  q - Rp\|)  \bigg) ^{1-a-b-c}\bigg] \\
 \forall p, q \in G \setminus \text{Fix}(R)
\end{aligned}
\end{equation} 
\end{definition}
\begin{lemma} {2.1}  \label{sec:lemma} \cite{bib2, bib3}  Let $\zeta \in \mathrm{Z}$ Then $\zeta(t) < t$ for all $t > 0$ and $\zeta(0) = 0$.
\end{lemma}
\begin{remark}
 \cite{bib24}
Considering a self-mapping $R$ on a linear space $G$, then for any $\lambda \in (0, 1)$, the so-called averaged mapping $R_\lambda$ given by
\begin{equation}\label{2.4}
R_\lambda p = (1 - \lambda)p + \lambda Rp, \quad \forall p \in G,
\end{equation}
has the property that $\text{Fix}(R_\lambda) = \text{Fix}(R)$.\\
\end{remark}

\section{Approximating fixed points of enriched  interpolative matkowski mapping}

\begin{definition}
Let $(G, \|\cdot\|)$ be a complete normed linear space. A self-mapping $R: G \to G$ is called an enriched interpolative  Matkowski mapping if there exist $a, b, c \in (0, 1)$ satisfying $a + b + c < 1$ , $k \in [0, +\infty)$, and $\zeta \in \mathrm{Z}$ such that\\

	\begin{equation}
		\begin{aligned}
			\| & k(p-q) + Rp - Rq \| \leq \zeta \Bigg[ \|(k+1)(p - q)\|^b \|p - Rp\|^a \|q - Rq\|^c \\
			&\left(\frac{1}{2} \left( \| (k+1)(p-q) + q - Rq\| + \| (k+1)(q-p) + p - Rp\| \right)\right)^{1-a-b-c} \Bigg]\\
			& \forall p, q \in G \setminus \text{Fix}(R)\\
			\end{aligned}
	\end{equation}
	\end{definition}
\begin{theorem}
Let $(G, \| \cdot \|)$ be a complete normed linear space and $R : G \rightarrow G$ an enriched interpolative  Matkowski mapping (3.1). Then
\begin{enumerate}
\item The mapping R has a  fixed point denoted as $s$.
\item For any given initial point $p_0$ belonging to the set $G$ there exists a real number $\lambda$ within the interval $(0, 1]$ such that the iterative scheme $\{p_m\}_{m=0}^\infty$, defined by the recurrence relation

\begin{equation}\label{4}
p_{m+1} = (1 - \lambda)p_m + \lambda Rp_m, \quad m \geq 0
\end{equation}

    converges to $s$, for any $p_0 \in G$.
    \end{enumerate}
\end{theorem}

\begin{proof}
If $k > 0$ in (3.1), We can choose $\lambda$ to be $\frac{1}{k + 1}$. It's evident that \\ $0 < \lambda < 1$ \\
  Consequently, equation (3.1) transforms into.
\[
\begin{aligned}
	\|& \bigg(\frac{1}{\lambda}-1 \bigg)(p-q) + Rp - Rq \| 
	\leq \zeta\Bigg[ \| \bigg(\frac{1}{\lambda}\bigg)(p - q)\|^b \|p - Rp\|^a \|q - Rq\|^c \\
	& \bigg(\frac{1}{2} (\| \bigg(\frac{1}{\lambda}\bigg)(p-q) + q - Rq\| + \| \bigg(\frac{1}{\lambda}\bigg)(q-p) + p - Rp\|)\bigg )^{1-a-b-c} \Bigg]
\end{aligned}
\]

\[
\begin{aligned}
\| (1-\lambda)(p-&q) +\lambda( Rp - Rq) \| 
\leq \lambda\zeta\bigg[ \bigg(\frac{1}{\lambda}\bigg)^{a+b+c}\|(p - q)\|^b \|\lambda(p - Rp)\|^a\\\|\lambda(q - Rq)\|^c 
& \bigg( \frac{1}{2} ( \| (1/\lambda)(p-q) + q - Rq\| + \| (1/\lambda)(q-p) + p - Rp\|)\bigg)   ^{1-a-b-c}\bigg]
\end{aligned}
\]\\

\[
\begin{aligned}
\|(1-\lambda)p +& \lambda Rp - (1-\lambda)q - \lambda Rq\| 
\leq \lambda\zeta\bigg[\bigg(\frac{1}{\lambda}\bigg)\| (p - q) \|^b\\
&\|(p - (1-\lambda)p - \lambda Rp)\|^a
\| (q - (1-\lambda)q - \lambda R(q))\|^c \\
&\bigg(\frac{1}{2}(\|p - (1-\lambda)q - \lambda R(q)\| + 
  \|q - (1-\lambda)p - \lambda R(p)\|\bigg)^{1-a-b-c}\bigg]
\end{aligned}
\]\\
\begin{equation}\label{3.3}
\begin{aligned}
\|((R_{\lambda} (p)- & R_{\lambda }(q))\|
\leq \lambda\zeta\bigg( \| \bigg(\frac{1}{\lambda}\bigg)\|(p - q)\|^{b} \|p-R_\lambda(p)\|^{a}
\|q -R_\lambda (q))\|^{c}\\
&\bigg( \frac{1}{2} ( \| p- R_\lambda(q)\| + 
\|(q - R_\lambda(p))\|)  \bigg) ^{1-a-b-c}\bigg]
\end{aligned}
\end{equation}

In accordance with  (2.2), the iterative sequence denoted as $\{p_m\}_{m=0}^\infty$, as defined by equation (3.2), is recognized as the Picard iteration associated with
$R_\lambda$, that is,
\[p_{m+1} = R_\lambda p_m, \quad m \geq 0, \quad m \in \mathbb{N}\]

Let $p_0 \in G$ be such that $p_m = R^m p_0$ for all $m \geq 0$. If $p_m = p_{m+1}$ for some $m$, then the conclusion remains valid.\\
If $p_m \neq p_{m+1}$, from (3.3), we have \\

Take $p = p_m$ and $q = p_{m-1}$ in equation (3.3) to get 
\begin{equation}\label{3.4}
\begin{aligned}
0 < \|((p_{m+1} - p_m))\|&
\leq \lambda\zeta\bigg[  \left(\frac{1}{\lambda}\right)\|(p_m - p_{m-1})\|^{b+c} \|p_m - p_{m+1}\|^a\\
&\bigg( \frac{1}{2} ( \|p_{m-1} - p_m\| + \|p_m- p_{m+1}\|)  \bigg) ^{1-a-b-c}\bigg]\ , m \geq 1
\end{aligned}
\end{equation}

Now we will show that $\|p_{m+1} - p_m\| \leq \|p_m- p_{m-1}\|$, $m \geq 1$. By contradiction, assume that \\$\|p_m - p_{m-1}\| \leq \|p_{m+1} - p_m\|$ for some $m$. \\
Putting this into (3.4) we have

$\|p_{m+1} - p_m\| \leq \lambda \zeta ( \left(\frac{1}{\lambda}\right) \|p_{m+1} - p_m\| )$\\

since $\zeta$ is an increasing function.
and also $\zeta(t) < t$ for all $t > 0$
,\\  so we have $\|p_{m+1} - p_m\| < \lambda \left(\frac{1}{\lambda}\right) \|p_{m+1} - p_m\|$\\
$\|p_{m+1} - p_m\| < \|p_{m+1} - p_m\|$ \\
Hence, we arrive at a contradiction. Therefore, the sequence $\|p_m - p_{m-1}\|$ is decreasing. Hence sequence $\|p_m - p_{m-1}\|$ is convergent .
 So from (3.4), we have

 \begin{equation*}
\begin{aligned}
\|p_{m+1} & - p_m\| \leq \lambda \zeta \bigg[ \left(\frac{1}{\lambda}\right) \|p_m - p_{m-1}\|^{b + c+ a }\|p_m - p_{m-1}\|^{1-a-b-c} \bigg]
\end{aligned}
\end{equation*}
\begin{equation*}
\begin{aligned}
\|p_{m+1} & - p_m\| \leq \lambda \zeta (\left(\frac{1}{\lambda}\right)\|p_m - p_{m-1}\|)
\end{aligned}
\end{equation*}
Similarly,
\begin{equation*}
\begin{aligned}
\|p_{m} & - p_{m-1}\| \leq \lambda \zeta (\left(\frac{1}{\lambda}\right)\|p_{m-1} - p_{m-2}\|)
\end{aligned}
\end{equation*}

On repeating the same argument, we get
\begin{equation*}
\begin{aligned}
\|p_{m+1} & - p_m\| \leq \lambda \zeta^{2}(\left(\frac{1}{\lambda}\right)\|p_{m-1} - p_{m-2}\|),m \geq 1
\end{aligned}
\end{equation*}\\
\begin{equation*}
\begin{aligned}
\|p_{m+1} & - p_m\| \leq \lambda \zeta^{m}(\left(\frac{1}{\lambda}\right)\|p_1 - p_{0}\|),m \geq 1
\end{aligned}
\end{equation*}\\
As $\zeta \in \mathrm{Z}$ and also we have  $\lim_{m\to\infty} \zeta^m(t) = 0$ for all $t > 0$\\ So it will imply that $\lim_{m\to\infty} \|p_{m-1} - p_m\| = 0$\
now our claim to prove that $\{p_m\}$ is a Cauchy sequence in $G$.
Let $\epsilon > 0$ be given. 
then there will exists $m_0 \in \mathbb{N}$ such that for all $m \geq m_0$ we have
\begin{equation}\label{eq:3.5}
\|p_m- p_{m+1}\| <  \epsilon
\end{equation} where
\begin{equation}\label{eq:3.6}
0 < \|p_m - p_{m +1}\| < \epsilon.
\end{equation}
Now state that  for all $n \geq m \geq m_0$:
We will utilize mathematical induction as our method of proof. We start by assuming the statement is valid when $n =k$. Now, let's examine equation (3.6) for 
$n=k+1$
We assume that it is true for $n = k$. Consider \eqref{eq:3.6} for $n = k + 1$.
For $k \geq n \geq m_0$ and using the induction hypothesis:
\begin{align*}
	\|p_m - p_{k+1}\| &\leq \|p_m - p_{m+1}\| + \|p_{m+1} - p_{k+1}\| \\
	&\leq \|p_m - p_{m+1}\| + \|R_{\lambda} p_m - R_{\lambda} p_k\| \\
 &\leq \|p_m - p_{m+1}\| +\lambda \zeta\bigg[ \left(\frac{1}{\lambda}\right) \|p_m - p_k\|^b \|p_m - p_{m+1}\|^a\\
&\|p_k - p_{k+1}\|^c \big( \frac{1}{2} (\|p_m - p_{k+1}\| + \|p_k - p_{m+1}\|) \bigg)^{1-a-b-c} \bigg]
\end{align*}
\[
\|p_m - p_{k+1}\| \leq \|p_m - p_{m+1}\| + \lambda \zeta\left(\frac{\epsilon}{\lambda}\right).
\]

Also, we have, \(\|p_m - p_{k+1}\| \leq \epsilon  + \epsilon = 2\epsilon\).

Thus, \eqref{eq:3.6} holds for \(m = k + 1\), and \(\{p_m\}\) is a Cauchy sequence in \(G\).
 and hence it is convergent in the complete normed linear space $(G, \|\cdot\|)$.\\
Let us denote
\begin{equation}\label{eq:3.8}
s = \lim_{m\to\infty} p_m.
\end{equation}

Now it remains to show that $s$ is a  fixed point of $R_\lambda$. We have

\[
\|s - R_\lambda s\| \leq \|s - p_{m+1}\| + \|p_{m+1} - R_\lambda s\| = \|s- p_{m+1}\| + \|R_{\lambda} p_m- R_{\lambda} s\|
\]
\begin{equation*}
\begin{aligned}
\|R_{\lambda} p_m - & R_{\lambda} s)\|
\leq \lambda \zeta \bigg[ \left(\frac{1}{\lambda}\right) \|p_m - s\|^b \|p_m - R_{\lambda} p_m\|^a \|s - R_{\lambda} s\|^c  \\
& \bigg(\frac{1}{2} \left(\|p_m - R_{\lambda} s\| + \|s - R_{\lambda} p_m\| \right) \bigg)^{1-a-b-c}\bigg]
\end{aligned}
\end{equation*}

\begin{equation*}
\begin{aligned}
\|R_{\lambda} p_m - & R_{\lambda} s\| \leq \lambda  \bigg[ \left(\frac{1}{\lambda}\right) \|p_m - s\|^b \|p_m - R_{\lambda} p_m\|^a \|s - R_{\lambda} s\|^c  \\
& \bigg( \frac{1}{2} \left(\|p_m - R_{\lambda} s\| + \|s - R_{\lambda} p_m\| \right) \bigg)^{1-a-b-c}\bigg]
\end{aligned}
\end{equation*}

Now, as we take the limit as $m \to \infty$, we obtain $\|s - R_\lambda s\| = 0$, which implies that  $s = R_\lambda s$.
 So,
$s \in \text{Fix} (R_\lambda)$
and since $\text{Fix}(R) = \text{Fix}(R_{\lambda})$, claim (1) is proven.\\
Second claim hold  by (3.8)
The remaining case $k = 0$ is the same as $k \neq 0$
with the only difference that in this case $\lambda = 1$, and hence, we work with $R = R_1$ when the Krasnoselskij iteration (2.2) reduces to the simple Picard iteration

The case when $k = 0$ is identical to the case when $k \neq 0$, except for $\lambda = 1$. Consequently, we are operating with $R = R_1$.
 In this case the Krasnoselskij iteration (2.2) simplifies into the standard Picard iteration:"
\[p_{m+1} = R p_m, \quad m \geq 0.\]
\end{proof}

\begin{example}
Let $G$= $M_2(\mathbb{R})$ be the collection of $2 \times 2$ square matrices, and define the norm $\| \cdot \|$ on $M_2(\mathbb{R})$ as follows,
\[
\| A \| = \max \{ |a_{ij}| \; | \; a_{ij} \text{ belongs to } A \}.
\]
Clearly, ($G, \|.\|$) is complete normed linear space .\\
Define $\zeta : [0, \infty) \to [0, \infty)$ as $\zeta(t) = \frac{2}{3}t$.\\
Now, let $R : G \to G$ be a mapping defined as $R(A) = -\frac{1}{4}A$, where $G$ is the set of matrices $A$ where $A$ belongs to $M_2(\mathbb{R})$. Let $k = \frac{1}{4}$.
\begin{equation*}
\begin{aligned}
    \|  &(k(p-q) + Rp - Rq) \| 
    \leq \zeta \bigg[ \| (k+1)(p - q) \|^{b} \|p - Rp\|^{a} \|q - Rq\|^{c} \\
   & \bigg( \frac{1}{2} \left( \| (k+1)(p-q )+ q - Rq\| + \| (k+1)(q-p) + p - Rp\| \right) \bigg)^{1-a-b-c} \bigg]
\end{aligned}
\end{equation*}
As $\| (k(p-q) + R(p)- R(q)) \| = 0$ for all $p, q\in M_2(\mathbb{R})$, \\it is evident that $R$ is an enriched interpolative Matkowski-type mapping.\\
Therefore,
by the fixed-point theorem, it has a fixed point as 
$\begin{bmatrix}
0 & 0 \\
0 & 0 \\
\end{bmatrix}$
\end{example}

\begin{corollary} Let $(G, \|\cdot\|)$ be complete normed linear space. We say that the self-mapping $R : G \to G$ is a self-mapping if there exist $a, b, c \in (0, 1)$ such that $a+b+c<1$ and $M \in [0, 1)$, $k \in [0, +\infty)$ such that\\
\[
\begin{aligned}
\| k(p-q) + Rp - Rq \| 
\leq M\bigg[ \| (k+1)(p - q)\|^b \|p - Rp\|^a \|q - Rq\|^c \\
 \bigg( \frac{1}{2} ( \| (k+1)(p-q) + q - Rq\| + \| (k+1)(q-p) + p - Rp\|)  \bigg) ^{1-a-b-c}\bigg]\\
     \forall p, q \in G \setminus \text{Fix}(R)
\end{aligned}
\]
Then $R$ has a fixed point in $G$.
\end{corollary}

\begin{corollary} Let $(G, \|\cdot\|)$ be a complete normed linear space. We say that the self-mapping $R: G \to G$  satisfy the condition of definition 2.2, Then $R$ has a fixed point in $G$.
\end{corollary}
\vspace{0.2cm}

\begin{corollary}
Any enriched interpolative Matkowski  mapping R satisfying 
with k=0 is also a  interpolative matkowski mapping . If k = 0 then from (3.1) we obtain the original interpolative matkowski type
mapping (2.1). Hence, any interpolative matkowski mapping is a (0, a,b,c)-enriched interpolative matkowski mapping.
\end{corollary}

\begin{remark}
	 It is clear that the enriched interpolative matkowski mapping (3.1) implies (2.1)
but the reverse is not true as shown by the next example implies (4.1)
\end{remark} 

\begin{example}
Let $G = \mathbb{R}^3$ and $R: \mathbb{R}^3 \to \mathbb{R}^3$ such that $R(z) = -\frac{1}{2} z$ where $z = (p, q, r)$, $k = \frac{1}{2}$, 
define $\zeta \colon [0, \infty) \rightarrow [0, \infty)$ such that $ \zeta(t) = \frac{t}{14}$ .
 and $\|.\| = |p| + |q| + |r|$. Clearly, it satisfies
\begin{equation*}
\begin{aligned}
    \| &k(x-y) +  Rx - Ry) \| \leq \zeta \bigg[ \| (k+1)(x - y) \|^{b} \|x - Rx\|^{a} \|y - Ry\|^{c} \\
    & \bigg( \frac{1}{2} \left( \| (k+1)(y-x) + x - Rx\| + \| (k+1)(x-y) + y - Ry\| \right) \bigg)^{1-a-b-c} \bigg]
\end{aligned}
\end{equation*}

as $\| (k(x-y) + Rx - Ry) \| = 0$ for all $x, y \in \mathbb{R}^3$ So clearly $R$ is an enriched interpolative Matkowski-type mapping and hence has a  fixed point as $(0,0,0)$.

However, the inequality does not satisfy the condition of interpolative matkowski mapping for all cases. \\For instance,
Let $b= \frac{1}{2}$, $a = c = \frac{1}{8}$, $x = (2, 2, 2)$, $y = (-2, -2, -2)$,\\ $Rx = (-1, -1, -1)$, and $Ry = (1, 1, 1)$.\\ 
We have $\|Rx - Ry\| = 6$ but 
\newpage
\begin{equation*}
\begin{aligned}
 \zeta \bigg( \| x - y \|^{b} \|x - Rx\|^{a} \|y - Ry\|^{c} 
 \left[ \frac{1}{2} \left( \| x - Ry\| + \| y - Rx\| \right) \right]^{1-a-b-c} \bigg) \\
  = \zeta([3.464][1.316][1.316][1.316])\\
  = \zeta(13.67664) =
 0.9769
 \\
\end{aligned}
\end{equation*}
So, the example given above satisfies the condition of an enriched interpolative Matkowski mapping, but it does not satisfy the condition of an interpolative Matkowski mapping. Therefore, the class of interpolative Matkowski mappings is strictly contained within the class of enriched interpolative Matkowski mapping.
\end{example}
\subsection{figure}
\begin{figure}[ht!]    \includegraphics[width=1\textwidth]{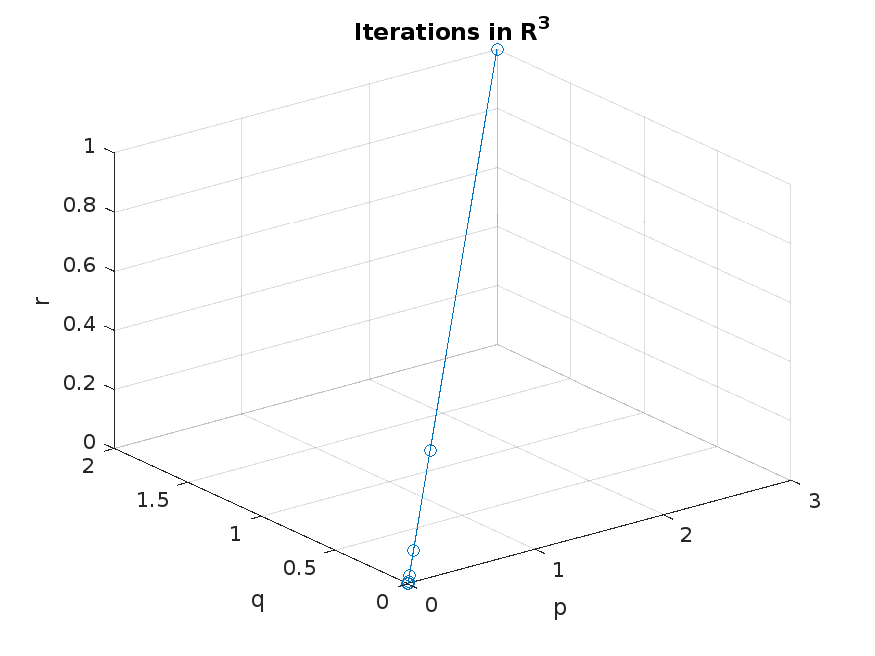}
    \caption{ Krasnoselskij Iteration for $T: \mathbb{R}^3 \to \mathbb{R}^3$ with $T(p, q, r) = -\frac{1}{2}(p, q, r)$}
    \label{fig:enter-label}
\end{figure}
\begin{table}[ht]
    \centering
    \caption{Results of the numerical experiments for $R (p,q,r) = -\frac{1}{2} (p,q,r)$
}
    \label{tab:data}
    \begin{tabular}{|c|c|c|c|}
    \hline
    \textbf{n} & \textbf{p} & \textbf{q} & \textbf{r} \\
    \hline
     0 & 3 & 2 & 1 \\
        1 & 0.75 & 0.5 & 0.25 \\
        2 & 0.1875 & 0.125 & 0.0625 \\
        3 & 0.046875 & 0.03125 & 0.015625 \\
        4 & 0.011719 & 0.0078125 & 0.0039062 \\
        5 & 0.0029297 & 0.0019531 & 0.00097656 \\
        6 & 0.00073242 & 0.00048828 & 0.00024414 \\
        7 & 0.00018311 & 0.00012207 & 6.1035e-05 \\
        8 & 4.5776e-05 & 3.0518e-05 & 1.5259e-05 \\
        9 & 1.1444e-05 & 7.6294e-06 & 3.8147e-06 \\
        10 & 2.861e-06 & 1.9073e-06 & 9.5367e-07 \\

    \hline
    \end{tabular}
\end{table}
In the above graph, we have denoted the Krasnoselskij Iteration in $\mathbb{R}^3$.\\ The sequence of iteration converges towards the fixed point $(0,0,0)$ for a particular value of $\lambda$. We have taken $\lambda = 0.5$ and $n=10$ iterations.\\
\newpage
\section{Numerical Experiments}
This section presents numerical experiments to demonstrate the authority of the Krasnoselskij algorithm,
\begin{equation}\label{eq:5.1}
p_{n+1} = Tp_n := (1 - \lambda)p_n + \lambda T(p_n)
\end{equation}
over the standard method of successive approximations,
\begin{equation}\label{eq:5.2}
q_{n+1} = Tq_n
\end{equation}
We shall write a very interesting fact about the effect of the control parameter. We shall perform iteration for different values of the parameter $\lambda$ on the rate of convergence of the sequence $x_n$ produced by the Krasnoselskij algorithm when analyzing the numerical tests presented in the following: The sequence $x_n$ converges towards the fixed point for different different  parameter $\lambda$.
\subsection{Example}
\begin{example}
\begin{figure}[ht!]
\includegraphics[width=1\textwidth]{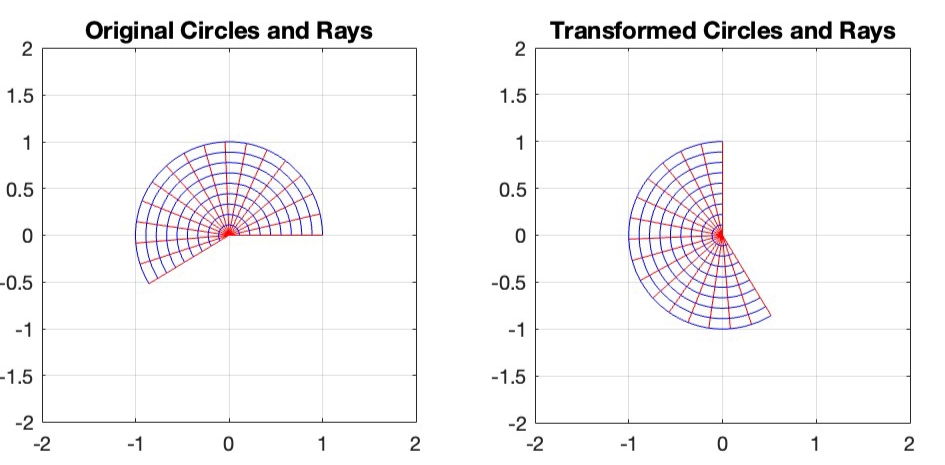}
\ \caption{ Domain and Range of Rotation map under the map $R:(x,y)=(-y,x)$} 
    \label{fig:enter-label}
\end{figure}
Take the example $R : D^2 \to D^2$,where $D^2$ denotes the closed unit ball in $\mathbb{R}^2$,Define as
\begin{equation}\label{eq:5.1}
R(re^{i\theta}) = re^{i(\theta + \frac{\pi}{2})} \
\end{equation} ,where $0 \leq \theta < 2\pi$.\\
Define $\|\cdot\|$ as the Euclidean norm on $D^2$. clearly 
(G,$\|\cdot\|$) is complete normed linear space.also we can define $R:(x,y)=(-y,x)$, where $x, y$ belong to $D^2$ where $D^2$ is the closed unit disc.\\
Furthermore, the Picard iteration ${y_n}$, defined by equation (4.2), exhibits non-convergence for any initial value $y_0$ distinct from the unique fixed point.
In contrast, the Krasnoselskij algorithm ${x_n}$, defined by equation (4.1), converges for all $\lambda \in (0, 1)$ and any initial value $x_0 \in G$.\\
In the figure given below, we have shown the rotation of an arc by the angle $\frac{\pi}{2}$ in the complex plane.\\
\begin{table}[h]
\centering
\caption{Results of the numerical experiments for  $R:(x,y)=(-y,x)$}
\begin{tabular}{|c|c|c|c|c|}
\hline
$n$ & $\lambda=0.1$ & $\lambda=0.2$ & $\lambda=0.3$ & $\lambda=0.4$   \\
\hline
\hline
0 & (0.5,1) & (0.5, 1) & (0.5, 1) & (0.5, 1) \\
1 & (0.35, 0.95) & (0.2, 0.9) & (0.05, 0.85) & (-0.1, 0.8) \\
2 & (0.22, 0.89) & (-0.02, 0.76) & (-0.22, 0.61) & (-0.38, 0.44) \\
3 & (0.10, 0.82) & (-0.16, 0.60) & (-0.33, 0.36) & (-0.40, 0.11) \\
4 & (0.01, 0.75) & (-0.25, 0.45) & (-0.34, 0.15) & (-0.28, -0.09) \\
5 & (-0.06, 0.67) & (-0.29, 0.30) & (-0.28, 0.002) & (-0.13, -0.17) \\
6 & (-0.12, 0.60) & (-0.29, 0.18) & (-0.20, -0.08) & (-0.01, -0.15) \\
7 & (-0.17, 0.53) & (-0.27, 0.09) & (-0.11, -0.11) & (0.05, -0.09) \\
8 & (-0.20, 0.46) & (-0.23, 0.01) & (-0.04, -0.11) & (0.07, -0.037) \\
9 & (-0.23, 0.39) & (-0.19, -0.03) & (0.003, -0.09) & (0.05, 0.006) \\
10 & (-0.25, 0.33) & (-0.14, -0.06) & (0.03, -0.06) & (0.03, 0.027) \\
\hline
\end{tabular}
\end{table}
\end{example}
In the table presented above, we illustrate the convergence behavior of the Krasnoselskij Iteration under the map $R: (p,q) \mapsto (-q,p)$ for various values of the control parameter $\lambda$. Specifically, we consider $\lambda =0.1, 0.2, 0.3, 0.4$. The results are obtained using an initial value of $(p_0, q_0) = (0.5,1)$.
Below, we provide a visual representation of the iterations corresponding to the parameter values $\lambda =0.1, 0.2, 0.3, 0.4 $ for $n$=$20$ iterations.

\begin{figure}[ht!]
    
    \includegraphics[width=1\textwidth]{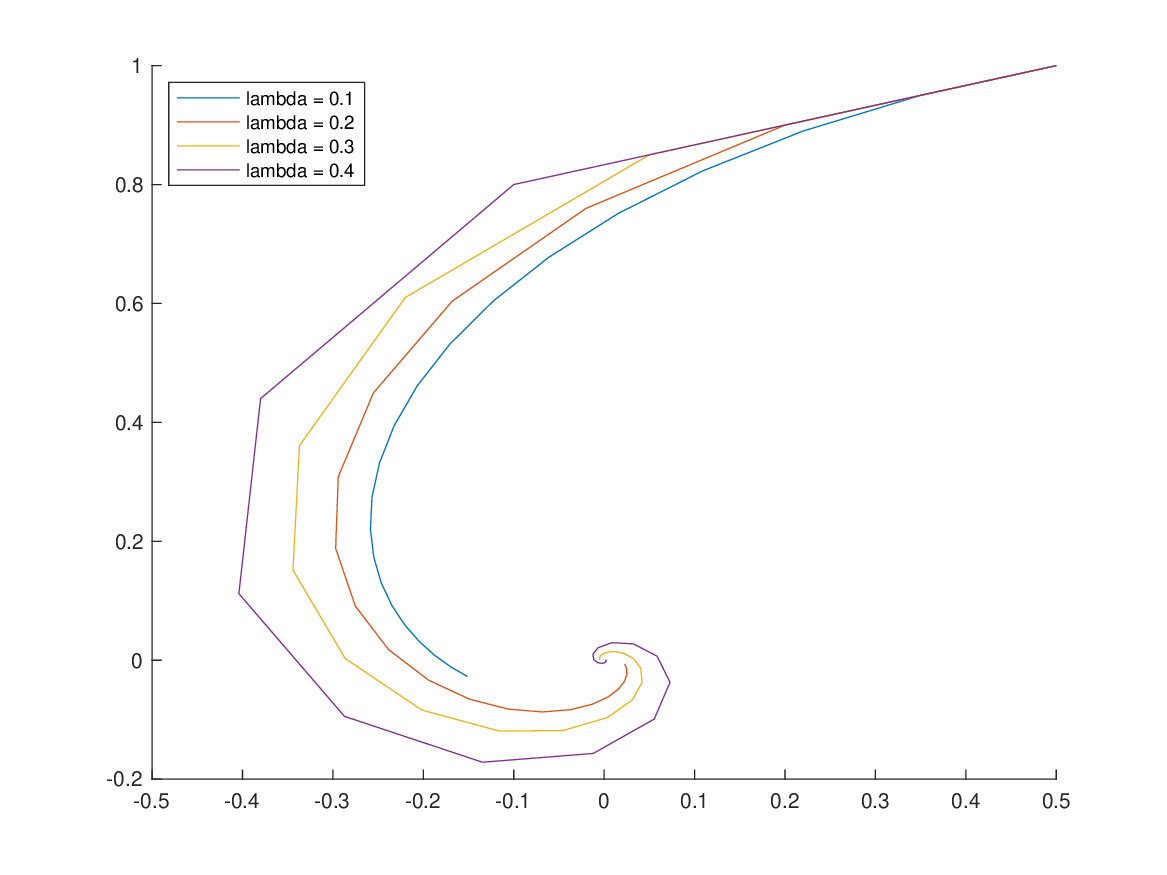}
    \caption{Convergence of iteration in $\mathbb{D}^2$
}
    \label{fig:enter-label}
\end{figure}

\subsection{Theorem}
\begin{theorem}
Let $(G, \|\cdot\|)$ be a complete normed linear space. Self-mappings $R, S: G \to G$ with the condition  
\begin{equation}
	\begin{aligned}
		\| k(p-q) + Rp - Sq\| \leq \zeta \bigg[ \| (k+1)(p - q) \|^b \|p - Rp\|^a \|q - Sq\|^c \\
	 \bigg( \frac{1}{2} ( \| (k+1)(p-q) + q - Sq\| + \| (k+1)(q-p) + p - Rp\|)\bigg) ^{1-a-b-c} \bigg]
	\end{aligned}
\end{equation}
for all $p, q \in G \setminus \text{Fix}(R),\text{Fix}(S)$
where $a, b, c \in (0, 1)$ satisfying $a + b + c < 1$, $k \in [0, +\infty)$, and $\zeta \in \mathrm{Z}$ then $R$ and $S$ have a common fixed point $s$.
\end{theorem}
\begin{proof}
		If $k > 0$ in (4.4), we can choose $\lambda$ to be $\frac{1}{k + 1}$. It's evident that $0 < \lambda < 1$
		Consequently, equation (4.4) transforms into.
		\begin{equation*}
			\begin{aligned}
				\|& \bigg(\frac{1}{\lambda}-1\bigg)  (p-q) + Rp - Sq \| 
				\leq \zeta\bigg[ \| \frac{1}{\lambda}p - q\|^b \|p - Rp\|^a \|q - Sq\|^c \\
				& \bigg( \frac{1}{2} (\| \left(\frac{1}{\lambda}\right)(p-q) + q - Sq\| + \| \left(\frac{1}{\lambda}\right)(q-p) + p - Rp\|)  \bigg) ^{1-a-b-c}\bigg]
			\end{aligned}
		\end{equation*}
		\begin{equation*}
			\begin{aligned}
				\|(1-\lambda)&(p-q)  +\lambda( Rp - Sq) \| \\
				&\leq \lambda\zeta\bigg[ (\frac{1}{\lambda})^{a+b+c}\|p - q\|^b \|\lambda (p - Rp)\|^a\|\lambda(q - Sq)\|^c \\
				& \bigg( \frac{1}{2} \bigg( \| (1/\lambda)(p-q) + q - Sq\| + \| (1/\lambda)(q-p) + p - Rp\|)  \bigg) ^{1-a-b-c}\bigg]
			\end{aligned}
		\end{equation*}
		\begin{equation*}
			\begin{aligned}
				&\|(1-\lambda)p + \lambda R(p) - (1-\lambda)q - \lambda S(q)\| \\
				&\leq \lambda\zeta\bigg[\left(\frac{1}{\lambda}\right)\| (p - q) \|^b\|(p - (1-\lambda)p - \lambda R(p))\|^a
				\| (q - (1-\lambda)q - \lambda S(q))\|^c \\
				&\bigg(\frac{1}{2}(\|p - (1-\lambda)q - \lambda S(q)\| + \|q - (1-\lambda)p - \lambda R(p)\|)\bigg)^{1-a-b-c}\bigg]
			\end{aligned}
		\end{equation*}

		\begin{equation}\label{3.3}
			\begin{aligned}
				\|((R_{\lambda} (p)- & S_{\lambda }(q))\|
				\leq \lambda\zeta\bigg[ \| \bigg(\frac{1}{\lambda}\bigg)\|(p - q)\|^{b} \|p-R_\lambda(p)\|^{a}
				\|q -S_\lambda (q))\|^{c}\\
				&\bigg( \frac{1}{2} ( \| p- S_\lambda(q)\| + 
					\|q - R_\lambda(p)\|)  \ ^{1-a-b-c}\bigg)\bigg]
			\end{aligned}
		\end{equation}\\
		Let $p_0 \in G$, define the sequence $p_m$ by
		\[
		p_{2m+1} = T_{\lambda}(p_{2m}), \quad p_{2m+2} = S_{\lambda}(p_{2m+1})
		\] 
		for all $m \in \mathbb{N} \cup \{0\}$.\\
		If $p_{2m} = p_{2m+1}=p_{2m+2}$ for some $m$, then $p_{2m}$ is common fixed point of $S_{\lambda}$ and $R_{\lambda}$.\\
		Suppose there do not exist three consecutive identical terms in the sequence \(p_m\).
		\begin{equation}\label{3.4}
			\begin{aligned}
				\|p_{2m+1} - p_{2m+2}\|&=\|R_{\lambda}{p_{2m}} -S_{\lambda}{p_{2m+1}}\|\\
				&\leq \lambda\zeta\bigg[\left(\frac{1}{\lambda}\right)\|(p_{2m} - p_{2m+1})\|^{b+c} \|p_{2m+1} - p_{2m+2}\|^a\\
				&\bigg( \frac{1}{2} ( \|p_{2m} - p_{2m+2}\| + \|p_{2m+1}- p_{2m+1}\|) \bigg) ^{1-a-b-c}\bigg]\ ,\\ 
				&\leq \lambda\zeta\bigg[ \left(\frac{1}{\lambda}\right)\|(p_{2m} - p_{2m+1})\|^{b+c} \|p_{2m+1} - p_{2m+2}\|^a\\
				&\bigg(\frac{1}{2} ( \|p_{2m} - p_{2m+1}\| + \|p_{2m+1}- p_{2m+2}\|)\bigg) ^{1-a-b-c}\bigg]\ , m \geq 1
			\end{aligned}
		\end{equation}

		Now we will show that $\|p_{2m+1} - p_{2m+2}\| \leq \|p_{2m}- p_{2m+1}\|$, $m \geq 1$. By contradiction, assume that \\$\|p_{2m}-p_{2m+1}\| \leq \|p_{2m+1}-p_{2m+2}\|$ for some $m$. \\
		Putting this into (4.6) we have
		
		$\|p_{2m+1} - p_{2m+2}\| \leq \lambda \zeta \bigg[ \left(\frac{1}{\lambda}\right) \|p_{2m+1} - p_{2m+2}\| \bigg]$\\
		
		since $\zeta$ is an increasing function.
		and also $\zeta(t) < t$ for all $t > 0$
		,\\  so we have $\|p_{2m+1} - p_{2m+2}\| < \lambda \left(\frac{1}{\lambda}\right) \|p_{2m+1} - p_{2m+2}\|$\\
		$\|p_{2m+1} - p_{2m+2}\| < \|p_{2m+1} - p_{2m+2}\|$ \\
		Hence, we arrive at a contradiction. Therefore, the sequence $\|p_{2m+1} - p_{2m+2}\|$ is decreasing. Hence sequence $\|p_{2m+1} - p_{2m+2}\|$ is convergent .
		So from (4.6), we have
		
		\begin{equation*}
			\begin{aligned}
				\|p_{2m+1} & - p_{2m+2}\| \leq \lambda \zeta \bigg[ \left(\frac{1}{\lambda}\right) \|p_{2m} - p_{2m+1}\|^{b + c+ a }\|p_{2m }- p_{2m+1}\|^{1-a-b-c} \bigg]
			\end{aligned}
		\end{equation*}
		\begin{equation*}
			\begin{aligned}
				\|p_{2m+2} & - p_{2m+1}\| \leq \lambda \zeta \bigg[(\frac{1}{\lambda})\|p_{2m }- p_{2m+1}\|\bigg]
			\end{aligned}
		\end{equation*}
		Similarly
		\begin{equation*}
			\begin{aligned}
				\|p_{2m+2} & - p_{2m+1}\| \leq \lambda \zeta \bigg[(\frac{1}{\lambda})\|p_{2m} - p_{2m+1}\|\bigg]
			\end{aligned}
		\end{equation*}
		
		On repeating the same argument, we get
		\begin{equation*}
			\begin{aligned}
				\|p_{2m+2} & - p_{2m+1}\| \leq \lambda \zeta^{2}\bigg[(\frac{1}{\lambda})\|p_{2m-1} - p_{2m+1}\|\bigg],m \geq 1
			\end{aligned}
		\end{equation*}\\
		\begin{equation*}
			\begin{aligned}
				\|p_{2m+1} & - p_{2m+2}\| \leq \lambda \zeta^{2m+1}(\left(\frac{1}{\lambda}\right)\|p_1 - p_{0}\|),m \geq 1
			\end{aligned}
		\end{equation*}\\
	
		To show that $p_m$ is Cauchy, let $m, k \in \mathbb{N}$.\\
		\begin{align*}
			\|p_m - p_{m+k}\| &\leq \|p_m - p_{m+1}\| + \|p_{m+1} - p_{m+2}\| + \cdots + \|p_{m+k-1} - p_{m+k}\| \\
			&\leq (\zeta^m + \zeta^{m+1} + \cdots + \zeta^{m+k-1}) \|p_0 - p_1\| = \sum_{i=m}^{m+k-1} \zeta^i \|p_0 - p_1\| \\
			&\leq \sum_{i=m}^{\infty} \zeta^i \|p_0 - p_1\|
		\end{align*}\\
		\begin{align*}
			\|&p_{m+n+k} - p_{m+k}\|\\ &\leq \|p_{m+n} - p_{m+n+1}\| + \|p_{m+n+1} - p_{m+n+2}\| + \cdots + \|p_{n+m+k-1} - p_{n+m+k}\| \\
			&\leq \sum_{i=n}^{\infty} \zeta^i \|p_{m} - p_{m+1}\|\\
			& =\lim_{n \to \infty} \sum_{i=n}^{\infty} \lim_{m \to \infty} \zeta^i \|p_{m} - p_{m+1}\|=0
		\end{align*}\\
		as $ \zeta^{m}{(t)} = 0 $\\
		$\lim_{m \to \infty} \|p_{m}- p_{m+k}\|=\lim_{m ,n\to \infty} \|p_{n+m}- p_{n+m+k}\|=0$\\
		\(\{p_m\}\) is a Cauchy sequence in \(G\).
		and hence it is convergent in the complete normed linear space $(G, \|\cdot\|)$.\\
		Let us denote
		\begin{equation}\label{eq:3.8}
			s = \lim_{m\to\infty} p_m.
		\end{equation}
		
		Now it remains to show that $s$ is a   common fixed point of fixed point of $S_{\lambda}$ and $R_{\lambda}$. We have

		\begin{equation*}
			\begin{aligned}
				\|R_{\lambda}s-  p_{2m+2}\|& =\|R_{\lambda}s-S_{\lambda}{p_{2m+1}}\|\\
				&\leq \lambda \zeta \bigg[ \left(\frac{1}{\lambda}\right) \|p_{2m+1} - s\|^b \|s- R_{\lambda}s\|^a \|p_{2m+1} - p_{2m+2}\|^c  \\
				& \bigg( \frac{1}{2} (\|s- p_{2m+2} \| + \|p_{2m+1} - R_{\lambda}s\|) \bigg)^{1-a-b-c}\bigg],
			\end{aligned}
		\end{equation*}

		Let \( m \) tend to infinity and, given that \(\|\cdot\|\) is a continuous function, we have \(\|R_{\lambda}u - u\| = 0\). Therefore, we conclude that \( R_{\lambda}s = s \).\\
		\begin{equation*}
			\begin{aligned}
				\|S_{\lambda}(s)-  p_{2m+1}\|& =\|R_{\lambda}p_{2m}-S_{\lambda}s\|\\
				&\leq  \zeta \bigg[ \left(\frac{1}{\lambda}\right) \|p_{2m} - s\|^b \|s - S_{\lambda}s\|^c \|p_{2m} - p_{2m+1}\|^a  \\
				&\bigg( \frac{1}{2} \left(\|S_{\lambda}s- p_{2m} \| + \|s - T_{\lambda}p_{2m}\| \right) \bigg)^{1-a-b-c}\bigg].
			\end{aligned}
		\end{equation*}\\
		Let \( m \) tend to infinity and, given that \(\|\cdot\|\) is a continuous function, we have \(\|S_{\lambda}s - s\| = 0\). Therefore, we conclude that \( S_{\lambda}s = s\).\\
		As \(s\) is a fixed point of \(R_{\lambda}\), \(s\) is also a fixed point of \(R\)\\
	   and also	\(s\) is a fixed point of \(S_{\lambda}\), \(s\) is also a fixed point of \(S\).
	  $R$ and $S$ have common fixed point.
	\end{proof}

	\section{Applications}
In this section, we will apply fixed point theory to address the split convex feasibility problem by using the enriched interpolative Matkowski mapping.
\begin{enumerate}
    \item $C, Q$ are nonempty, closed, and convex subsets of a real Hilbert space $\mathcal{H}$.
\item $\mathcal{H}$ is a real Hilbert space
.
\item $T: \mathcal{H} \rightarrow \mathcal{H}$ bounded linear  operator.
\item $T^*: \mathcal{H} \longrightarrow \mathcal{H}$ be the adjoint of $T$.
\item $P_c, P_Q$ denote  projection operators on $C$ and $Q$, respectively.
\item The Operator \(L: \mathcal{H} \rightarrow \mathcal{H}\) is defined as follows:
\begin{equation}\label{eq:6.1}
Lp = P_C\left(p + \frac{1}{\|T\|^2} T^*\cdot \left(P_Q(Tp) - Tp\right)\right).
\end{equation}
\end{enumerate}
\begin{theorem} \textbf{5.1}\cite{bib23,bib25}
Let \(C\) and \(Q\) be nonempty, closed, and convex subsets of real Hilbert spaces \(\mathcal{H}\). Let \(\mathcal{T}: \mathcal{H} \rightarrow \mathcal{H}\) be a bounded linear operator, and let \(T^*: \mathcal{H} \rightarrow \mathcal{H}\) be its adjoint. The split convex feasibility problem is formulated as follows:

Find an element \(x^* \in C\) such that \(Tx^* \in Q\).\\
\end{theorem}
\textbf{Result} \textbf{5.2}  \cite{bib26,bib27}
 If \(T: \mathcal{H} \rightarrow \mathcal{H}\) is a bounded linear operator and \(C\) and \(Q\) are closed and convex subsets of \(\mathcal{H}\), Let $K$ be the solution set of the Split Feasibility Problem. We assume that $K$ is nonempty. Then, for any $x^* \in C$, $x^*$ is a solution of Split Feasibility Problem if and only if it solves the following fixed-point equation 
\(
P_C \left[ I - \frac{1}{\|T\|^2}T^*(I - P_Q)T \right]x = x, \quad x \in C,
\)\\

\begin{theorem}
Consider the operator \(L\) defined in equation (\ref{eq:6.1}) satisfying the\\ condition
\begin{equation}\label{eq:6.2}
\begin{aligned}
\| & k(x-y)+  Lx-Ly\| \\
&\leq \zeta\bigg[ \| (k+1)(x - y)\|^b \|x - Lx\|^a \|y - Ly\|^c \\
&  \bigg
( \frac{1}{2} ( \| (k+1)(x-y) + y - Ly\| + \| (k+1)(y-x) + x- Lx\|) \bigg)^{1-a-b-c} \bigg] ,\\
& \forall x, y \in  \mathcal{H} \setminus \text{Fix}(L)
\end{aligned}
\end{equation}
where $a, b, c \in (0, 1)$ satisfying $a + b + c < 1$, $k \in [0, +\infty)$, and $\zeta \in \mathrm{Z}$ \
then there exist  \(x^* \in C\) such that \(Tx^* \in Q\).
\end{theorem}
\begin{proof}
 Using Theorem (3.2), \(L\) has a  fixed point, and using Result (5.2), the fixed point of operator \(L\) is equivalent to finding \(x^* \in C\) such that \(Tx^* \in Q\).
\end{proof}
\section{Conclusion}
In this paper construction of enriched interpolative Matkowski mappings is defined, which encompasses interpolative Matkowski \cite{bib4} mappings as a particular case. It is proven that every enriched interpolative Matkowski mapping possesses a  fixed point that can be estimated using Krasnoselski iterations. Examples are provided to illustrate that the class of enriched interpolative Matkowski mappings  is a strict superset of interpolative Matkowski mappings\cite{bib4}. In other words, there exist mappings that do not belong to the set of interpolative Matkowski mappings but do belong to the set of enriched interpolative Matkowski mappings. In Section 5, numerical experiments are presented to complement our theoretical results.

\end{document}